\documentclass[a4paper,10pt,oneside]{article}
\usepackage{latexsym,amsfonts,amssymb,amsmath}
\usepackage{amsthm}

\usepackage[english]{babel}

\usepackage{diagrams}
\usepackage{marginnote}
\usepackage{color}
\usepackage{comment}
\usepackage[hmargin=1in, vmargin=0.85in]{geometry}

\DeclareMathOperator{\codim}{codim}

\DeclareMathOperator{\Spec}{Spec}

\DeclareMathAlphabet{\mathpzc}{OT1}{pzc}{m}{it}

\newcommand{\mbb}{\mathbb}

\newcommand{\mc}{\mathcal}

\newcommand{\FS}{\mathcal{O}}

\newcommand{\id}{{\rm id}}

\newcommand{\Tor}{{\rm Tor}}

\newcommand{\ra}{\rightarrow}
\newcommand{\perm}{\mathfrak{S}}

\newcommand{\trest}{{\big |} }

\newcommand{\mf}{\mathfrak}

\newcommand{\tens}{\otimes}

\newtheorem{theorem}{Theorem}[section]
\newtheorem{lemma}[theorem]{Lemma}
\newtheorem{pps}[theorem]{Proposition}
\newtheorem{crl}[theorem]{Corollary}

\newtheorem{conj}{Conjecture}

\theoremstyle{definition}
\newtheorem{definition}[theorem]{Definition}

\theoremstyle{remark}
\newtheorem{remark}[theorem]{Remark}

\numberwithin{equation}{section}

\mathchardef\phi="0127
\mathchardef\varphi="011E

\mathchardef\alpha="710B
\mathchardef\beta="710C
\mathchardef\gamma="710D
\mathchardef\delta="710E
\mathchardef\epsilon="7122
\mathchardef\zeta="7110
\mathchardef\eta="7111
\mathchardef\theta="7112
\mathchardef\iota="7113
\mathchardef\kappa="7114
\mathchardef\lambda="7115
\mathchardef\mu="7116
\mathchardef\nu="7117
\mathchardef\xi="7118
\mathchardef\pi="7119
\mathchardef\rho="711A
\mathchardef\sigma="711B
\mathchardef\tau="711C
\mathchardef\upsilon="711D
\mathchardef\chi="711F
\mathchardef\psi="7120
\mathchardef\omega="7121
\mathchardef\varepsilon="710F
\mathchardef\vartheta="7123
\mathchardef\varpi="7124
\mathchardef\varrho="7125
\mathchardef\varsigma="7126

\theoremstyle{definition}

\theoremstyle{remark}

\numberwithin{equation}{section}

\DeclareMathOperator{\lct}{lct}
\DeclareMathOperator{\fpt}{fpt}
\DeclareMathOperator{\ord}{ord}

\DeclareMathOperator{\Proj}{Proj}

\renewcommand{\div}{\mathrm{div} \,}
\newcommand{\wh}{\widehat}

\usepackage{hyperref}

\title{Singularities of the Isospectral Hilbert Scheme}
\author{Luca Scala}
\date{}

\begin{document}
\maketitle

\begin{abstract}We study the singularities of the isospectral Hilbert scheme $B^n$ 
of $n$ points over a smooth algebraic surface and we prove 
that they are canonical if $n \leq 5$, log-canonical if $n \leq 7$ and not log-canonical if $n \geq 9$. We describe as well two explicit log-resolutions 
of $B^3$, one crepant and the other $\perm_3$-equivariant. 
\end{abstract}

\normalsize
\section*{Introduction}
The aim of this work is the study 
of the singularities of the isospectral Hilbert scheme of $n$ points over a smooth complex algebraic surface. If $X$ is such a surface,   
the isospectral Hilbert scheme $B^n$  can be defined as the blow-up  of the product variety $X^n$ along the big diagonal $\Delta_n$. The isospectral Hilbert scheme has been introduced  by Haiman in his works \cite{Haiman1999} and \cite{Haiman2001} on Macdonald polynomials; it was proven in \cite{Haiman2001} that $B^n$ is normal, Cohen-Macauley and Gorenstein. 
Haiman himself asked in \cite[Section 1]{Haiman2004} whether the Rees algebra $\oplus_{i \geq 0} \mc{I}^i_{\Delta_n}$ were of $F$-rational type; this would be equivalent
of $\Spec \oplus_{i \geq 0} \mc{I}^i_{\Delta_n}$ having rational singularities \cite{Smith1997, Hara1998, MehtaSrinivas1997, SchwedeTakagi2008} and would imply \cite[Proposition 1.2]{Hyry1999} that $B^n = \Proj ( \oplus_{i \geq 0} \mc{I}^i_{\Delta_n} )$
would have rational or, equivalently, canonical singularities. 
It is an open problem whether $B^n$ has canonical or log-canonical singularities. 
In this work we partially answer these questions. 

Apart from being  interesting in its own, the investigation of the singularities of $B^n$ is in tight relation with a number of interesting problems. The first and more immediate --- which is one of the main motivations of this work --- is the potential application 
  to vanishing theorems, since sufficiently good singularities would allow the use of Kawamata-Viehweg or Kodaira vanishing over $B^n$; 
  an example of this use already appeared  in \cite[Section 5.2]{Scalaarxiv2015}. 
  
  A second source of interest, which also offers an effective way to address the problem, is the link with the study of 
  log-canonical thresholds of subspace arrangements. 
 Since $B^n$ is the blow-up of the big diagonal in $X^n$, it turns out that the scheme
$B^n$ --- or, in other words, the pair $(B^n, \emptyset)$ --- has exactly the same kind of singularities of the pair
$(X^n, \mc{I}_{\Delta_n})$.  Now, one can determine the kind of singularities of the pair $(X^n, \mc{I}_{\Delta_n})$ by studying 
its  log-canonical threshold at each point. Since this problem is now local in nature, one can take $X$ as the affine plane $\mbb{C}^2$: in this case the big diagonal $\Delta_n$ can be thought as a subspace arrangement. This problem is similar with that of finding log-canonical thresholds of hyperplane arrangements, already studied and solved in \cite{Mustata2006}. On the other hand, there are not many examples in literature of computations of log-canonical thresholds of  arrangements of subspaces of higher codimension: an exception is the study of configurations of lines through the origin in $\mbb{C}^3$ by Teitler \cite{Teitler2007}.  An important part of his work deals with the understanding of the embedded components that appear when pulling back the ideal of the configuration of lines to the blow-up of the origin in $\mbb{C}^3$; the presence of embedded components is the main difficulty that hinders an explicit log-resolution of the ideal of the configuration. 

The case of the pair  $(X^n, \mc{I}_{\Delta_n})$ ---  for $X = \mbb{C}^2$ --- is similar because we deal with an arrangement of codimension 2 subspaces $\Delta_n$ in $\mbb{C}^{2n}$, but  the complexity of the problem grows very rapidly with $n$. However, for $X = \mbb{C}^2$, Haiman gave a precise description of a set of generators for the ideal $\mc{I}_{\Delta_n}$, from which we can deduce the order of the ideal $\mc{I}_{\Delta_n}$ at each point. As a consequence, we can establish the upper bound (proposition \ref{pps: lctinequality}) 
$$ \lct (X^n, \mc{I}_{\Delta_n}) \leq \frac{2n - 2}{d_n} $$for the log-canonical threshold of the pair $(X^n, \mc{I}_{\Delta_n})$. Here 
$d_n$ is the natural number defined in remark~\ref{rmk: mingenerators}.  We actually believe that the above inequality is in fact an equality (Conjecture \ref{conj1}). This would imply that the singularities of $B^n$ are canonical if and only if $n \leq 7$, log-canonical if $n \leq 8$ and not log-canonical if $n \geq 9$ (Conjecture 2). We can actually prove --- and this is the main result of this work~---

\vspace{0.2cm} \noindent
{\bf Theorem \ref{thm: logcanonical}}. {\it The singularities of the isospectral Hilbert scheme $B^n$ are canonical if $n \leq 5$ and log-canonical if $n \leq 7$. 
For $n \geq 9$ they are not log-canonical.} 

\vspace{0.2cm}
Not unexpectedly, this problem  is in close relation with the geometry of the Hilbert scheme of points as well. 
Indeed, after a result by Song in \cite{SongThesis}, results about the pair $(X^n, \mc{I}_{\Delta_n})$ 
can be precisely translated into results about the pair $(X^{[n]}, \mc{I}_{\partial X^{[n]}} )$, where $X^{[n]}$ is the Hilbert scheme of $n$ points over $X$ and $\partial X^{[n]}$ is its boundary. In particular the previous upper bound for $\lct(X^n, \mc{I}_{\Delta_n})$ implies the  upper bound
$\lct (X^{[n]}, \mc{I}_{\partial X^{[n]}}) \leq (n-1)/d_n$. The mentioned conjecture on $\lct(X^n, \mc{I}_{\Delta_n})$ 
would imply that the last upper bound is actually an equality.

Finally, the problem of understanding the singularities of the isospectral Hilbert scheme should be a drive to the construction of an explicit $\perm_n$-equivariant log-resolution of $B^n$, or --- what is equivalent --- to an explicit $\perm_n$-equivariant log-resolution $f: Y \rTo X^n$ of the pair $(X^n, \mc{I}_{\Delta_n})$. This would be a deep and importat result on many levels. Firstly, it would provide another important compactification of the configuration space $F(X, n) := X^n \setminus \Delta_n$ after the celebrated Fulton-MacPherson compactification $X[n]$ (see \cite{FultonMacPherson1994}): the latter is not, unfortunately, a log-resolution of the pair $(X^n, \mc{I}_{\Delta_n})$, since, when computing the inverse image of the ideal 
$\mc{I}_{\Delta_n}$ over $X[n]$ embedded components appear. Hence an explicit $\perm_n$-equivariant log-resolution of $(X^n, \mc{I}_{\Delta_n})$ might be built by further blowing-up the Fulton-MacPherson compactification in order to get rid of these components; however, it is a very difficult problem to track and control the embedded components that arise in this way. 

Secondly, supposing that the stabilizers of the $\perm_n$-action on the resolution 
$Y$ were trivial, then, passing to the quotient would provide an explicit resolution $\hat{f}: Y/\perm_n \rTo S^n X$ of the symmetric variety. We mention that, in general, no such explicit resolution is known yet. In \cite{Ulyanov2002} Ulyanov made a step forward proposing a refinement of the Fulton-MacPherson compactification in a way that the stabilizers  of the natural $\perm_n$-action are abelian, and not just solvable. 

Finally, such a resolution $f: Y \rTo X^n$ might be useful for a better understanding of  ideal sheaves of 
subschemes supported in  big diagonals of the form $\FS(-\lambda \Delta)$, appeared in the work \cite{Scalaarxiv2015}.

In the final section of this article we provide two different log-resolutions of the pair $(X^3, \mc{I}_{\Delta_3})$, and hence of $B^3$: one crepant, 
the other $\perm_3$-equivariant. 

We work over the field of complex numbers. By point we always mean a closed point.

\paragraph{Acknowledgements.} I would like to sincerely 
thank Lei Song for inviting me to University of Kansas, for his interest in my work and for 
communicating the results mentioned in subsection \ref{relationHilbert}. This work is partially supported by CNPq,  grant 307795/2012-8. 

\section{Singularities of pairs and log-canonical thresholds}
\begin{definition} \cite{Kollar1997, LazarsfeldPAGII} Let $M$ be an irreducible complex algebraic variety, and $\mf{a}$ an ideal sheaf of $\FS_M$. A \emph{log-resolution} of 
the pair $(M, \mf{a})$ is a projective birational map $f : Y \rTo M$ such that $Y$ is nonsingular,  the exceptional locus $\mathrm{exc}(f)$ is a divisor, the ideal sheaf $f^{-1} \mf{a} := \mf{a} \cdot \FS_Y$ is equal  to $\FS_Y(-F) $, where $F$ is an effective divisor on $Y$ with the property that $F + \mathrm{exc}(f)$ has simple normal crossing support. 
\end{definition}

\begin{definition}\label{dfn: cansingularities}Let $M$ be a  complex algebraic variety, normal and irreducible; let $K_M$ be its canonical divisor. Suppose that $M$ is $\mbb{Q}$-Gorenstein, that is, for some $r \in \mbb{N}^*$, 
$r K_M$ is Cartier. Let $\mf{a}$ be an ideal sheaf of $\FS_M$. Consider a log-resolution $f: Y \rTo M$ of the pair $(M, \mf{a})$. 
Then, as $\mbb{Q}$-Cartier divisors,  $$K_Y - f^*(K_M) + f^{-1}(\mf{a}) = \sum_{i} a_i E_i \;,$$where $E_i$ are irreducible component of a simple normal crossing divisor and $a_i \in \mbb{Q}$. 
We say that the singularities of the pair $(M, \mf{a})$ are  \emph{canonical} if $a_i \geq 0$; 
 \emph{log-canonical} if $a_i \geq -1$. 
\end{definition}

\begin{definition}Let $M$ be a smooth algebraic variety and $\mf{a}$ an ideal sheaf of $\FS_M$. Let $c \in \mbb{Q}$, $c >0$. Let $f: Y \rTo M$ be a log-resolution of the pair $(M, \mf{a})$ and let $F$ be the effective Cartier divisor on $Y$ such that $f^{-1} \mf{a} = \FS_Y(-F)$. Then the \emph{multiplier ideal sheaf} 
$\mc{J}(c \cdot \mf{a})$ associated to $c$ and $\mf{a}$ is the ideal sheaf of $\FS_M$ defined as
$$ \mc{J}(c \cdot \mf{a}) : = f_* \FS_Y(K_{Y/M} - [c \cdot F]) \;,$$where $[c \cdot F]$ is the integral part of the $\mbb{Q}$-divisor $F$. 
The definition just given does not depend on the choice of the log-resolution \cite{LazarsfeldPAGII}. 
For $x \in M$, the \emph{log-canonical threshold} of the pair $(M, \mf{a})$ at the point $x$ is defined as
$$ \lct_x (M, \mf{a}) := \sup \{ c \in \mbb{Q} \; | \; \mc{J} (c \cdot \mf{a})_x  = \FS_{M, x} \} = \inf \{ c \in \mbb{Q} \; | \; \mc{J} (c \cdot \mf{a})_x  \subset \mf{m}_x \} \;.$$Define, moreover, $\lct(M, \mf{a}) := \inf_{x \in M} \lct_x(M, \mf{a})$. 
\end{definition}
\begin{remark}\label{rmk: lctminima}In the above definition of $\lct_x(M, \mf{a})$ the $\inf$ are actually minima \cite[Example 9.3.16]{LazarsfeldPAGII}. 
\end{remark}

\begin{pps}\label{pps: blowup}Let $M$ be a smooth complex algebraic variety and let $\mf{a}$ be an ideal sheaf of $\FS_M$. Consider the blow-up $g: B := \mathrm{Bl}_{\mf{a}} M \rTo M$ of $M$ along the ideal $\mf{a}$, with exceptional divisor $E$. 
Suppose that $B$ is irreducible, normal and Gorenstein; 
suppose moreover that $K_B = g^* K_M + \FS_B(E)$. 
Then 
$B$ has (log-) canonical singularities if and only if the pair $(M, \mf{a})$ has. 
\end{pps}
\begin{proof}Let $h: Y \rTo B$ be a log-resolution of the pair $(B, E)$. Consider the map $f = g \circ h$. We claim that $f$ is a log-resolution of the pair $(M, \mf{a})$. Indeed 
$\mathrm{exc}(f)$ is divisorial, since $f$ is a birational morphism between smooth varieties. Moreover, set-theoretically, 
$\mathrm{exc}(f) = \mathrm{exc}(h) \cup h^{-1}\mathrm{exc}(g) =  \mathrm{exc}(h) \cup h^{-1} E$, which --- since $h$ is a log-resolution of $(B, E)$ --- is a  divisor with snc support. Hence $\mathrm{exc}(f)$ is a divisor with snc support. Moreover 
$f^{-1} \mf{a} = h^{-1} g^{-1} \mf{a} = h^{-1} \mc{I}_E = \FS_B(-h^*E)$ and $h^*E$ is an effective  Cartier divisor. Finally, as Cartier divisors, $\mathrm{exc}(f) + h^{*} E$ coincides 
 with $\mathrm{exc}(h) + 2 h^{*} E$, which has the same support as $\mathrm{exc}(f)$ and hence is a divisor with snc support. 
Then 
\begin{align*} K_Y - h^* K_B = \: &  K_Y - h^*g^*K_M - h^*\FS_B(E)   
 =  K_Y - f^*K_M + f^{-1} \mf{a}
 \end{align*}which allows us to conclude. 
\end{proof}
\section{The isospectral Hilbert scheme}
\begin{definition}Let $n \in \mbb{N}$, $n \geq 2$. Let $X$ be a smooth complex algebraic surface. Let $\Delta_n$ be the big diagonal in $X^n$, that is, $\Delta_n$ is the scheme-theoretic union of pairwise diagonals $\Delta_{ij}$, $1 \leq i < j \leq n$. The \emph{isospectral Hilbert scheme $B^n$} is 
the blow up of $X^n$ along the big diagonal $\Delta_n$. 
\end{definition}
\begin{remark}It is well known that the isospectral Hilbert scheme $B^n$ is irreducible, normal, Cohen-Macaulay and Gorenstein \cite{Haiman2001}.
\end{remark}
\subsection{The big diagonal in $X^n$}
As an immediate consequence of proposition \ref{pps: blowup}, we have a very precise correspondence between the singularities of the isospectral Hilbert scheme $B^n$ and those of the pair $(X^n, \mc{I}_{\Delta_n})$. 
\begin{crl}\label{crl: isospectral}The isospectral Hilbert scheme $B^n$ has (log-) canonical singularities if and only if the pair $(X^n, \mc{I}_{\Delta_n})$ has (log-) canonical singularities. 
\end{crl}
\begin{remark}\label{rmk: lctlogcanonical}It is well known \cite[Example 9.3.16]{LazarsfeldPAGII} that a pair $(M, \mf{a})$ 
has log-canonical singularities 
if and only if $\lct(M, \mf{a} ) \geq 1$. On the other hand, if $M$ is Gorenstein, then the discrepancies $a_i$ in definition \ref{dfn: cansingularities} are necessarily integers; consequently the pair $(M, \mf{a})$ is canonical if and only if $\lct( M, \mf{a}) >1$, that is, if and only if $\mc{J}(M, \mf{a}) = \FS_M$. 
Hence we have that the isospectral Hilbert scheme $B^n$ has canonical singularities if and only if $\lct(X^n, \mc{I}_{\Delta_n}) >1$ or, equivalently, if $\mc{J}(X^n, \mc{I}_{\Delta_n})$ is trivial; the singularities of $B^n$ are log-canonical if and only if $\lct(X^n, \mc{I}_{\Delta_n}) \geq 1$. 
\end{remark}
\begin{remark}\label{rmk: Cn}The log-canonical threshold $\lct_x(M, \mf{a})$ at the point $x \in M$ coincides with the \emph{complex singularity exponent $c_x(\mf{a})$}  of $\mf{a}$ at the point $x$ \cite{DemaillyKollar2001}, which is an holomorphic invariant. 
As a consequence, the log-canonical threshold  of the pair $(X^n, \mc{I}_{\Delta_n})$ for an arbitrary smooth algebraic surface $X$ is equal to the log-canonical threshold of the pair $((\mbb{C}^2)^n, \mc{I}_{\Delta_n})$. 
\end{remark}
\begin{remark}[Generators of $\mc{I}_{\Delta_n}$ for $X = \mbb{C}^2$]\label{rmk: generators} In \cite{Haiman2001} Haiman finds an explicit set of generators for ideal of
the big diagonal $\Delta_n$ of $(\mbb{C}^2)^n$. Write $(\mbb{C}^2)^n$ as $\Spec \mbb{C}[x_1, y_1, \dots, x_n, y_n]$. If $\bar{p}, \bar{q} \in \mbb{N}^n$, denote with $\Delta(\bar{p}, \bar{q}, \bar{x}, \bar{y})$ the $\perm_n$-anti-invariant regular function 
$$ \Delta(\bar{p}, \bar{q}, \bar{x}, \bar{y}) := \det (x_i^{p_j} y_i^{q_j})_{ij} $$in the variables $x_1, \dots, x_n, y_1, \dots, y_n$. 
If there is no risk of confusion, we will drop the indication of the variables and we will just write it as $\Delta(\bar{p}, \bar{q})$. Haiman proves that homogeneous polynomials of the form 
$\Delta(\bar{p}, \bar{q})$ generate the ideal $\mc{I}_{\Delta_n}$. Of course the function $\Delta(\bar{p}, \bar{q})$ is non identically zero if 
and only if the points $(p_i, q_i) \in \mbb{N} \times \mbb{N}$ are all distinct. 
\end{remark}
\begin{remark}[Generators of minimal degree in $\mc{I}_{\Delta_n}$] \label{rmk: mingenerators}A nonzero homogeneous polynomial of the form $\Delta(\bar{p}, \bar{q})$ 
is of minimal degree if the set of points $\{ (p_i, q_i), i = 1, \dots, n \}$ minimize the weight $\sum_i (p_i + q_i)$. Now for any $n \in \mbb{N}$  there exist two  natural numbers $k$ and $h$, with $h < k$, uniquely determined by $n$, such that $ n = k(k+1)/2 + h$. The integers $k$ and $h$ explain how to arrange $n$ distinct points $(p_i, q_i)$
in $\mbb{N} \times \mbb{N}$ in such a way that the weight $\sum_i (p_i+ q_i)$ is the minimum possible: fill in the first antidiagonals in $\mbb{N} \times \mbb{N}$, of weight
$0$ to $k-1$, 
with  $k(k+1)/2$ points of nonnegative integral coordinates and on the  antidiagonal of weight $k$  put, in an arbitrary way, $h$ points. Consequently, a generator of minimal degree has degree 
$$ d_n = \sum_{i=0}^{k-1} i(i+1) + hk = \frac{1}{3}k(k^2  +3h-1) \;.$$
\end{remark}
\begin{remark}\label{rmk: changevar} Consider the diagonal $\Delta_n$ inside $(\mbb{C}^2)^n = \Spec \mbb{C}[x_1, y_1, \dots, x_n, y_n]$ and consider its ideal $\mc{I}_{\Delta_n} \subseteq \mbb{C}[x_1, y_1, \dots, x_n, y_n]$. We build now a new coordinate system, in the following way. 
Consider the vector space $ (\mbb{C}^2)^{n-1}$ with coordinates $(z_1, w_1, \dots, z_{n-1}, w_{n-1})$ and   $\mbb{C}^2$ with coordinates $(\alpha, \beta)$. 
Consider now the isomorphism
\begin{equation}\label{eq: change} \phi: (\mbb{C}^2)^n \rTo (\mbb{C}^2)^{n-1} \times  \mbb{C}^2 
\end{equation}defined by the coordinate change 
\begin{align*} 
&z_i = x_1 -x_{i+1} \;, \quad  & w_i & = y_1 - y_{i+1} \hspace{1cm} &  \mbox{for $i = 1, \dots n -1$} \\
& \alpha = \sum_{i=1}^n x_i \;, \quad & \beta & = \sum_{i=1}^n y_i \;.\hspace{1cm} & 
\end{align*}In the new coordinates the pairwise diagonals in $(\mbb{C}^2)^n$ are now given by ideals $(z_i, w_i)$ and $(z_i -z_j, w_i -w_j)$, $1 \leq i <j \leq n-1$
and
the ideal $\mc{I}_{\Delta_n}$ is the intersection $$ \mc{I}_{\Delta_n} = \cap_{i=1}^{n-1}(z_i, w_i) \bigcap \cap_{1 \leq i < j \leq n-1} (z_i - z_j, w_i -w_j)$$inside $\mbb{C}[z_1, w_1, \dots, z_{n-1}, w_{n-1}, \alpha, \beta]$.  
Since the generators of $\mc{I}_{\Delta_n}$ are just polynomials in the $z_i, w_i$, the ideal $\mc{I}_{\Delta_n}$ is the extension of an ideal $\mc{I}_{\widetilde{D}_{n-1}} \subseteq \mbb{C}[z_1, \dots, z_{n-1}, w_1, \dots, w_{n-1}] $, generated by the same elements. 
In other words, we can write 
\begin{equation}\label{eq: identification}
\mc{I}_{\Delta_n} \simeq \phi^* (\mc{I}_{\widetilde{D}_{n-1}}  \boxtimes \FS_{\mbb{C}^2}) \;.
\end{equation}\sloppy
Consider now the projection 
$r: (\mbb{C}^2)^{n-1} \times \mbb{C}^2 \rTo (\mbb{C}^2)^{n-1}$.  
Under the identification $\phi$, 
the small diagonal 
$\Delta_{1, \dots, n}$ in $(\mbb{C}^2)^n$ is the pre-image $r^{-1}( \{ 0 \})$ by $r$ of the origin $\{ 0 \}$ in $(\mbb{C}^2)^{n-1}$. 
Consequently, the order of the big diagonal 
$\Delta_n$ along the small diagonal $\Delta_{1, \dots, n}$ coincide with the order of $\widetilde{D}_{n-1}$ at the origin: $\ord_{\Delta_{1, \dots, n}} \mc{I}_{\Delta_n} = \ord_0 \mc{I}_{\widetilde{D}_{n-1}}$; but $\ord_0 \mc{I}_{\widetilde{D}_{n-1}}$ is the minimal degree of generators of $\mc{I}_{\widetilde{D}_{n-1}}$. But $\mc{I}_{\Delta_n}$ and $\mc{I}_{\widetilde{D}_{n-1}}$ have the same generators, hence $\ord_{\Delta_{1, \dots , n }} \mc{I}_{\Delta_n} = d_n$. Since the order of a coherent ideal along a subvariety is an holomorphic invariant, we can say in general that, for a smooth algebraic surface~$X$, $$ \ord_{\Delta_{1, \dots, n}} \mc{I}_{\Delta_n} = d_n \;.$$
\end{remark}

\begin{remark}Consider $X = \mbb{C}^2$. Note that, if $\{ (p_i, q_i), i =1, \dots, n-1 \}$ is a set of $n-1$ distinct points in $\mbb{N} \times \mbb{N}$ not containing the origin, the polynomial $\Delta(\bar{p}, \bar{q}, \bar{z}, \bar{w})$ belongs to $\mc{I}_{\widetilde{D}_{n-1}}$. 
\end{remark}

\subsection{$F$-pure thresholds}
For computational convenience we consider the characteristic $p$ analogue of the log-canonical threshold \cite{TakagiWatanabe2004, MTW2005}. Let $k$ be a perfect field of characteristic $p$;  let $R$ be a finitely generated regular $k$-algebra and  $\mf{a} \subseteq R$ a nonzero ideal; 
consider $M= \Spec R$ and let $x \in V(\mf{a})$ be a closed point corresponding to a maximal ideal $\mf{m}_x$. For $e \in \mbb{N}^*$, define
$$ \nu_{\mf{a}}(e) := \max \left \{ i \in \mbb{N} \; | \; \mf{a}^i \not \subseteq \mf{m}_x^{[p^e]} \right \} $$where $\mf{m}_x^{[p^e]}$ is the ideal generated by $p^e$-powers of generators of $\mf{m}_x$. The inequality 
$\nu_{\mf{a}}(e+1) \geq p \nu_{\mf{a}}(e)$ implies that the sequences $\nu_{\mf{a}}(e)/p^e$ and $\nu_{\mf{a}}(e)/(p^e -1)$ are nondecreasing \cite[Lemma 1.1]{MTW2005}. The
$F$-pure threshold of the ideal $\mf{a}$ at the point $x$ is defined as 
\begin{equation}\label{eq: deffpt} \fpt_{x}(M, \mf{a}):= \lim_{e \ra + \infty} \frac{\nu_{\mf{a}}(e)}{p^e} = \lim_{e \ra + \infty} \frac{\nu_{\mf{a}}(e)}{(p^e-1)} = \sup_{e \in \mbb{N}^*} \frac{\nu_{\mf{a}}(e)}{p^e}   = \sup_{e \in \mbb{N}^*} \frac{\nu_{\mf{a}}(e)}{(p^e-1)} \;.\end{equation}
Suppose now that $\mf{a}$ is principal: we write simply $\nu_f(e)$ instead of $\nu_{(f)}(e)$ and $\fpt_{x}(M, f)$ instead of $\fpt_{x}(M, (f))$. 
In this case the sequence 
$\nu_{\mf{a}}(e)/p^e$ is bounded above by $1$.  
Hence, for any $e \in \mbb{N}^*$ we have the inequalities
\begin{equation}\label{eq: fptinequality} \frac{\nu_{f}(e)}{(p^e-1)} \leq \fpt_x(M,f) 
\leq 1 \;.\end{equation}

Suppose now that $M$ is the affine space $\mbb{A}^n_{\mbb{Z}}$ over $\mbb{Z}$ and $\mf{a}$ is a nonzero ideal of $R := \mbb{Z}[x_1, \dots, x_n]$. For any prime $p$ consider the mod $p$ reduction $M_p := \Spec(R \tens_{\mbb{Z}} \mbb{F}_p)$ and 
$\mf{a}_p = \mf{a} \cdot \mbb{F}_p[x_1, \dots, x_n]$. On the other hand, if $\mbb{K}$ is an arbitrary field extension of $\mbb{Q
}$ we can consider the extensions $\mf{a}_{\mbb{K}}$
inside $\mbb{K}[x_1, \dots, x_n]$, respectively  and $M_{\mbb{K}} : = \Spec(R \tens_{\mbb{Z}} \mbb{K})$. For varieties defined over arbitrary perfect fields, Zhu recently proved an interpretation of the log-canonical threshold in terms of dimensions of jet-schemes \cite[Theorem B]{Zhu2013arxiv};  this result yields, as a consequence, the inequality $\fpt_x(M_p, \mf{a}_p) \leq \lct_x(M_{\mbb{Q}}, \mf{a}_{\mbb{Q}}) $
 for every prime $p$ and for every closed point $x \in V(\mf{a})$ \cite[Corollary 4.2]{Zhu2013arxiv}. Since the dimension of a scheme does not change upon extension of the field of definition \cite[Corollaire 4.1.4]{EGAIV2}, we have, for every prime $p$ and any closed point $x \in V(\mf{a})$ \begin{equation}\label{eq: fptlct} \fpt_x(M_p, \mf{a}_p) \leq \lct_x(M_{\mbb{C}}, \mf{a}_{\mbb{C}}) \;.\end{equation}

\subsection{Singularities of the isospectral Hilbert scheme}

We begin by establishing the following upper bound for the log-canonical threshold of the pair $(X^n, \mc{I}_{\Delta_n})$. 
\begin{pps}\label{pps: lctinequality}The log-canonical threshold of $(X^n, \mc{I}_{\Delta_n})$ is bounded above by $(2n-2)/d_n$: 
$$ \lct(X^n, \mc{I}_{\Delta_n}) \leq \frac{2n-2}{d_n} \;.$$
\end{pps}
\begin{proof}By remark \ref{rmk: Cn} it is sufficient to prove the inequality  when $X = \mbb{C}^2$. 
By remark \ref{rmk: changevar}, for $c \in \mbb{Q}$, $c >0$, 
the order of $c \cdot \mc{I}_{\Delta_n}$ along the small  diagonal $\Delta_{1, \dots, n}$ is $c d_n$; as soon as 
 $c d_n  \geq \codim_{X}\Delta_{1, \dots, n} + 1-1 = 2n-2$, that is, if $c \geq (2n-2)/d_n$, by  \cite[Example 9.3.7]{LazarsfeldPAGII} we have that $\mc{J}(X, c \cdot \mc{I}_{\Delta_n}) \subseteq \mc{I}_{\Delta_{1, \dots, n}}$. By definition of log-canonical threshold $\lct_0(X^n, \mc{I}_{\Delta_n})$ as infimum, we get the desired inequality $\lct(X^n, \mc{I}_{\Delta_n}) \leq \lct_0(X^n, \mc{I}_{\Delta_n}) \leq (2n-2)/d_n$. \end{proof}

\begin{remark}Consider the symmetric variety $S^nX$, where $X$ is a smooth complex algebraic surface; 
we will indicate with $\pi: X^n \rTo S^nX$ the quotient projection. 
It is well known 
that $S^nX$ admits a stratification in strata $S^n_\lambda X$, where $\lambda$ is a partition of $n$. 
The stratum $S^n_\lambda X$ is the locally closed subset of $0$-cycles of the form $\sum_{i=1}^{l(\lambda)} \lambda_i x_i$, where 
$l(\lambda)$ is the length of the partition $\lambda$ and $x_i$ are $l(\lambda)$ distinct points in $X$. By means of this stratification of $S^n X$ we can define a stratification of $X^n$ setting the stratum $X^n_\lambda$ as the locally closed subset
$\pi^{-1}(S^n_\lambda X)$. It is clear that if $x \in X^n_\lambda$ then a sufficiently small open set $V_1$ of $x$ in $X^n$ in the standard topology is biholomorphic to a sufficiently small open set $V_2$ of the origin in $(\mbb{C}^2)^n$ of the form $V_2 = U_1^{\lambda_1} \times \cdots \times U_{l(\lambda)}^{\lambda_{l(\lambda)}}$, where $U_i$ are adequate small open sets of the origin in $\mbb{C}^2$, 
such that, via the biholomorphic map,  the ideal $\mc{I}_{\Delta_n}$ over $V_1$ is sent to $\mc{I}_{\Delta_{\lambda_1}} \boxtimes \cdots \boxtimes \mc{I}_{\Delta_{\lambda_{l(\lambda)}}}$ over $V_2$.  Therefore, if $x \in X^n_\lambda$, we have, by proposition   \ref{pps: lctinequality} and by \cite[Proposition 9.5.22]{LazarsfeldPAGII} that 
\begin{equation} \label{eq: lctmin} \lct_x( X^n, \mc{I}_{\Delta_n}) = \min \big \{ \lct_0((\mbb{C}^2)^{\lambda_i}, 
\mc{I}_{\Delta_{\lambda_i}} )  \; | \; i = 1, \dots, l(\lambda) \big \} \leq \frac{2 \lambda_1 -2}{d_{\lambda_1} }\;.
\end{equation}
\end{remark}
We now make the following conjecture 
\begin{conj}\label{conj1}Let $X$ be a smooth algebraic surface. If a point $x$ of $X^n$ lies in the stratum $X^n_{\lambda}$, where $\lambda$ is a partition of $n$, then 
$ \lct_x(X^n, \mc{I}_{\Delta_n}) =  (2 \lambda_1 -2)/d_{\lambda_1} $. Therefore $$ \displaystyle \lct(X^n, \mc{I}_{\Delta_n}) = \frac{2n-2}{d_n} \;.$$
\end{conj}This conjecture would immediately imply the following fact about the singularities of the isospectral Hilbert scheme $B^n$. 
\begin{conj}\label{conj: seconda}The singularities of the isospectral Hilbert scheme $B^n$ are canonical if and only if $n \leq 7$, log-canonical if $n \leq 8$, 
not log-canonical if $n \geq 9$. 
\end{conj}

We are able to partially prove conjecture \ref{conj: seconda}. 

\begin{theorem}\label{thm: logcanonical}The singularities of the isospectral Hilbert scheme $B^n$ are canonical if $n \leq 5$, log-canonical if $n \leq 7$. 
For $n \geq 9$ they are not log-canonical. 
\end{theorem}
\begin{proof}By corollary \ref{crl: isospectral} and by remark \ref{rmk: lctlogcanonical} the singularities of the isospectral Hilbert scheme $B^n$ are log-canonical if and only if $\lct(X^n, \mc{I}_{\Delta_n})~\geq~1$ and canonical if and only if $\lct(X^n, \mc{I}_{\Delta_n})~>~1$. 
For $n \geq 9$, by proposition \ref{pps: lctinequality}, $\lct(X^n, \mc{I}_{\Delta_n}) \leq (2n-2)/d_n \leq 16/17$. Hence they can't be log-canonical.  

Let's now prove the first statement. Using corollary \ref{crl: isospectral} and remark \ref{rmk: lctlogcanonical} it is sufficient to prove that the singularities of the pair  $(X^n, \mc{I}_{\Delta_n})$ are canonical  for $n \leq 5$ and that $\lct(X^n, \mc{I}_{\Delta_n}) \geq 1 $ for $n =6,7$. 
By remark \ref{rmk: Cn} it is sufficient to prove these facts for $X = \mbb{C}^2$.  
By (\ref{eq: identification}), it is then sufficient to prove that the pair $(\mbb{C}^{2n-2}, \mc{I}_{\widetilde{D}_{n-1}})$ has canonical singularities for $n \leq 5$ and is log-canonical for $n =6,7$. 

To prove that the pair $(\mbb{C}^{2n-2}, \mc{I}_{\widetilde{D}_{n-1}} )$ is canonical for $n \leq 4$ we will use Kollar-Bertini theorem \cite[Theorems 4.5, 4.5.1]{Kollar1997}, \cite[Example 9.3.50]{LazarsfeldPAGII}: in other words we will find a 
$g \in \mc{I}_{\widetilde{D}_{n-1}}$ such that $\div g$ has rational (or canonical) singularities; then Kollar-Bertini theorem implies that the pair $(\mbb{C}^{2n-2},  \mc{I}_{\widetilde{D}_{n-1}})$ is canonical. For $n=3$ such a $g$ can be chosen as the generator of minimal degree 
of $\mc{I}_{\widetilde{D}_2}$, that is,  $g = z_1 w_2 -z_2 w_1$: it defines an affine quadric cone of in $\mbb{C}^4$ projecting a smooth  quadric in $\mbb{P}^3$ from the origin of $\mbb{C}^4$. 
Hence, by \cite[Example 1.2]{Burns1974}, it has rational singularities. For $n =4$ we can use the generator of minimal degree of $\mc{I}_{\widetilde{D}_3}$ 
given by  the polynomial $g = \Delta((1,0,1), (0,1,1), \bar{z}, \bar{w})$.  
One can show that $g$ 
has rational singularities using \emph{Macaulay2} \cite{M2} and, in particular, the command {\tt hasRationalSing} of the package {\tt D-modules}.

For $n \geq 5$ it is computationally more efficient to use characteristic $p$ methods. Let now $n~=~5$.  By the equality in (\ref{eq: lctmin}) and by what we just proved, we know that for any point $x$ in a strata $X^{5}_\lambda$, with $\lambda \neq (5)$, we have 
$\lct_x(X^5, \mc{I}_{\Delta_5}) \geq \lct(X^4, \mc{I}_{\Delta_4}) >1$.  
 It is then sufficient to prove that, for a point $x \in \Delta_{1, \dots, 5}$,  
$\lct_x(\mbb{C}^{10}, \mc{I}_{\Delta_5})>1$. Because of the isomorphism (\ref{eq: identification}) it is sufficient to prove that $\lct_0(\mbb{C}^{8}, \mc{I}_{\widetilde{D}_4})  > 1$. By (\ref{eq: fptlct}) it is sufficient to prove, for some prime $p$, that $\fpt_0( (\mbb{F}_p^{2})^{4}, (\mc{I}_{\widetilde{D}_{4}})_p) > 1$. Consider the polynomials
 $g = \Delta((1,0,2,1),(0,1,0,2), \bar{z}, \bar{w})$ and $h = \Delta((1,0,2,0),(0,1,0,2), \bar{z}, \bar{w})$ in $\mc{I}_{\widetilde{D}_{4}}$; we can check, using  
 \emph{Macaulay2} and passing modulo $p=7$, that the class of $g^2 h^5$ is nonzero in 
 $\mbb{F}_7[z_1, \dots, z_{4}, w_1, \dots, w_{4}]/ \mf{m}_0^{[7]}$, thus proving that $\nu_{\mf{a}}(1) \geq 7$, where $\mf{a} = 
 (\mc{I}_{\widetilde{D}_{4}})_7$, and hence that $\fpt_0( (\mbb{F}_7^{2})^{4}, (\mc{I}_{\widetilde{D}_{4}})_7) \geq 7/6 >1$, by (\ref{eq: deffpt}).  Therefore the pair $(X^5, \mc{I}_{\Delta_5})$ has canonical singularities.

Let now $n = 6,7$. \sloppy
By the equality in (\ref{eq: lctmin}) and by what we just proved,  we already know that for any point $x$ in a stratum $X^{n}_\lambda$, with $\lambda \neq (6)$ --- in the case $n=6$ --- or $\lambda \neq (7)$ and
$\lambda \neq (6,1)$ --- in the case $n=7$ --- we have $\lct_x(X^n, \mc{I}_{\Delta_n}) \geq \lct(X^5, \mc{I}_{\Delta_5}) > 1$.  For $n=6$ it is then sufficient to prove that 
$\lct_x(\mbb{C}^{12}, \mc{I}_{\Delta_6}) \geq 1$ when $x \in \Delta_{1, \dots, 6}$; by the isomorphism (\ref{eq: identification}), it is sufficient to prove that $\lct_0(\mbb{C}^{10}, \mc{I}_{\widetilde{D}_5})  \geq 1$; once we prove it, it is sufficient to prove 
that $\lct_x(\mbb{C}^{14}, \mc{I}_{\Delta_7})>1$ for $x \in \Delta_{1, \dots, 7}$, or equivalenty, after (\ref{eq: identification}), that 
$\lct_0(\mbb{C}^{12}, \mc{I}_{\widetilde{D}_6})  \geq 1$. By (\ref{eq: fptlct}) it is sufficient to prove, for some prime $p$, that 
$\fpt_0( (\mbb{F}_p^{2})^{n-1}, (\mc{I}_{\widetilde{D}_{n-1}})_p) \geq 1$ for $n=6,7$. By the first of the inequalities (\ref{eq: fptinequality}) it is then sufficient to find a polynomial $g \in \mc{I}_{\widetilde{D}_{n-1}}$, with integral coefficients, such that, for some prime $p$,  
$\nu_{g_p}(1) = p-1$ at the origin: here, for a polynomial $g$ with integral coefficients,  we denote with $g_p$ its\!\!$\mod p$ reduction in  in  $
(\mc{I}_{\widetilde{D}_{n-1}})_p$.  Consider the polynomials with integral coefficients $g = \Delta((1,0,2,1,0), (0,1,0,1,2), \bar{z}, \bar{w})$, for $n=6$, and 
$h = \Delta((1,0,2,1,0,2),(0,1,0,1,2,1), \bar{z}, \bar{w})$, for $n=7$. Then, passing modulo $p=7$, we checked, using \emph{Macaulay2},  
that the classes of 
$g^6_7$ in $\mbb{F}_7[z_1, \dots, z_{5}, w_1, \dots, w_{5}]/ \mf{m}_0^{[7]}$ and $h^6_7$ in $\mbb{F}_7[z_1, \dots, z_{6}, w_1, \dots, w_{6}]/ \mf{m}_0^{[7]}$ are both non zero. This proves that, choosing the prime $p=7$,  $\nu_{g_7}(1) = 6 = \nu_{h_7}(1)$ and we can conclude.
\end{proof}

\subsection{Relation with the geometry of the Hilbert scheme of points}\label{relationHilbert}
The geometry of the pair $(X^n, \mc{I}_{\Delta_n})$ is not only directly related to the geometry of the isospectral Hilbert scheme $B^n$, but also 
to the geometry of the Hilbert scheme of $n$ points $X^{[n]}$ over the surface $X$. Consider the boundary $\partial X^{[n]}$ of $X^{[n]}$. 
Song proved in \cite[Proposition  4.3.5]{SongThesis} that 
$$ \lct (X^{[n]}, \mc{I}_{\partial X^{[n]}}) = \lct(S^n X, \mc{I}_{\Delta_n}^{\perm_n}) = \frac{1}{2} \lct (X^n , \mc{I}_{\Delta_n} ) \;.$$Hence proposition \ref{pps: lctinequality} implies immediately the 
\begin{crl}The log-canonical threshold of the pair $(X^{[n]}, \mc{I}_{\partial X^{[n]}})$ is bounded above by $(n-1)/d_n$. 
\end{crl}Moreover, conjecture \ref{conj1} would imply 
\begin{conj}The log-canonical threshold of the pair $(X^{[n]}, \mc{I}_{\partial X^{[n]}})$ is precisely given by $(n-1)/d_n$. 
\end{conj}

\section{Two resolutions of $B^3$}
The aim of this subsection is two provide two explicit resolutions of singularities of $B^3$; the first will be \emph{crepant}, the second will be \emph{$\perm_3$-equivariant}. We begin with some remarks and technical lemmas. 
\begin{remark}Let $M$ a smooth algebraic variety and let $F$ be a coherent sheaf over $M$. We recall that an integral subscheme 
$V$ of $M$ is called a \emph{prime cycle associated to $F$} if there exists an invertible coherent $\FS_V$-module $L$ and an embedding 
 $L \rInto F$ of coherent $\FS_M$-modules. 
 \end{remark}

\begin{remark}Let $M$ be a smooth algebraic variety and $Y$ a smooth subvariety. Let $Z \subseteq M$ be a closed subscheme, defined by the ideal sheaf $\mc{I}_Z$. 
Let $r = \ord_Y \mc{I}_Z$ the order of $Z$ along $Y$. Consider the blow-up $f: \mathrm{Bl}_Y M \rTo M$ of $Y$ in $M$ and denote with $E$ its exceptional divisor. 
The \emph{weak transform} $\widetilde{Z}$ of $Z$ in $\mathrm{Bl}_Y M$ is defined by the residual ideal 
$\mc{I}_{\widetilde{Z}} := (\mc{I}_{f^{-1}(Z)}: \mc{I}_{E}^r)$. The ideal of the total transform $f^{-1}(Z)$ is then given by the product 
$$ \mc{I}_{f^{-1}(Z)} = \mc{I}_{E}^r  \cdot \mc{I}_{\widetilde{Z}} \;.$$It is well known that the weak transform does not necessarily coincide with the 
strict transform $\widehat{Z}$; in general one just has that $\mc{I}_{\widetilde{Z}} \subseteq \mc{I}_{\widehat{Z}}$, and that the two ideals coincide outside the exceptional divisor. Indeed the weak transform $\widetilde{Z}$ could contain embedded components over the exceptional divisor, while the strict transform doesn't. 
This is, in any case, the only possible difference between $\widetilde{Z}$ and $\widehat{Z}$,  as the next criterion proves. 
\end{remark}
\begin{pps}\label{pps: lambda}Let $M$ be a smooth algebraic variety and $Y$ a smooth subvariety. Let $Z \subseteq M$ be a closed subscheme.
Consider the blow-up map $f: \mathrm{Bl}_{Y} M \rTo M$ and let $E$ be the exceptional divisor. Then the weak transform $\widetilde{Z}$ of $Z$ coincide with the strict transform $\widehat{Z}$ if and only if $E$ does not contain any prime cycle associated to $\widetilde{Z}$. In this case, for any positive integer $l$,  the subschemes $lE$ and $\widehat{Z}$ are transverse. 
\end{pps}
\begin{proof}The necessity of the condition is clear. We just have to prove the sufficiency. Recall that the strict transform $\widehat{Z}$ can be identified with the blow-up $\mathrm{Bl}_{Y \cap Z} Z$: this is a consequence, for example, of  \cite[Proposition IV-21]{EisenbudHarrisGS}. 
Indicate with $\lambda$ the canonical section of $\FS_{\mathrm{Bl}_Y M}(E)$. We have that $E$ does not contain prime cycles 
associated to $\widetilde{Z}$ if and only if the morphism $\lambda: \FS_{\widetilde{Z}}(-E) \rTo \FS_{\widetilde{Z}}$ is injective.  In this case the ideal $\mc{I}_{\widetilde{Z} \cap E / \widetilde{Z} }$ of $\widetilde{Z} \cap E $ in $\widetilde{Z}$ is an invertible ideal of $\FS_{\widetilde{Z}}$. Hence 
the map $f \trest_{\widetilde{Z}}: \widetilde{Z} \rTo Z$ factors via the blow-up $\mathrm{Bl}_{Y \cap Z}  Z$, that is, via the strict transform $\widehat{Z}$. Hence we have the injection of schemes $\widetilde{Z} \rInto \widehat{Z}$. But it is always true that $\widehat{Z} \subseteq \widetilde{Z}$. Hence the weak transform coincides with the strict one. In this case, for any fixed positive integer $l$,  the morphism $\lambda^l: \FS_{\widehat{Z}}(-lE) \rTo \FS_{\widehat{Z}}$ is injective. 
Since $R^\bullet := 0 \rTo \FS_{\mathrm{Bl}_Y M}(-lE) \rTo \FS_{\mathrm{Bl}_Y M}$ is a locally free resolution of $\FS_{lE}$, we can compute $\Tor_{j}(\FS_{lE}, \FS_{\widehat{Z}})$ as of the $(-j)$-cohomology of the complex $R^\bullet \tens \FS_{\widehat{Z}}$, which is 
$0 \rTo  \FS_{\widehat{Z}}(-lE) \rTo^{\lambda^l} \FS_{\widehat{Z}} \rTo 0$. 
Hence 
$\Tor_{j}(\FS_{lE}, \FS_{\widehat{Z}}) = 0$ for $j>0$. 
\end{proof}

\begin{remark}\label{rmk: H}Let $M$ be a smooth algebraic variety, and  $Y$  a smooth subvariety. 
Consider the blow-up map $f: \mathrm{Bl}_Y M \rTo M$. Let $H$ be an hypersurface in $M$.  
Then its weak and strict transform in $\mathrm{Bl}_Y M$ coincide. \end{remark}
\begin{proof}Let $E$ be the exceptional divisor. The weak transform $\widetilde{H}$ is a divisor whose associated prime cycles are the irreducible components of $\widetilde{H}$. Since, by definition of $\widetilde{H}$, one has that $E \not \subset \widetilde{H}$, then $\codim_{\mathrm{Bl_Y M}} E \cap \widetilde{H} = 2$ and hence the local equations of $E$ and $\widetilde{H}$ define a regular sequence; hence $E$ does not contain any prime cycles relative to $\widetilde{H}$. Hence $\widetilde{H} = \widehat{H}$. 
\end{proof}

\begin{lemma}Let $M$ be a smooth algebraic variety and let $Y, W, Z$ three subschemes of $M$, such that 
$Y$ is closed, $W$ is integral and that $Y \not \subseteq W$. Let $\widehat{W}$, $\widehat{Z}$ be the strict transforms of $W$ 
and $Z$ inside $\mathrm{Bl}_Y M$. Then $\ord_W \mc{I}_Z = \ord_{\wh{W}} \mc{I}_{\wh{Z}}$. 
\end{lemma}
\begin{proof}Note that if $S, T$ are two subschemes of a smooth algebraic variety $V$, with $T$ integral, then 
$\ord_T \mc{I}_S$ can be characterized as $ \ord_T \mc{I}_S = \max \{ n \in \mbb{N} \; | \; \mc{I}_{S, T} \subseteq \mf{m}_T^n  \}$ where 
$\mf{m}_T$ is the maximal ideal  of the local ring $\FS_{V, T}$ ---  that is, the ring of regular functions $g$ defined on some open set $U$ intersecting $T$ \cite[Exercise 3.13]{HartshorneAG}--- and where $\mc{I}_{S, T}$ is the ideal of functions $g$ in $\FS_{V, T}$ vanishing over $S \cap U$, if $U$ is the open set of definition of $g$. Now  the blow-up map $f: \mathrm{Bl}_Y M \rTo M$ induces an isomorphism of local rings 
$f_W^* : \FS_{M, W} \rTo \FS_{\mathrm{Bl}_Y M, \wh{W} }$ under which $\mc{I}_{Z, W}$ is sent onto $\mc{I}_{\wh{Z}, \wh{W}}$, hence the statement. 
\end{proof}
\begin{lemma}\label{lmm: w1w2}Let $M$ be a smooth algebraic variety of dimension at least $3$; let $H$ be a smooth hypersurface in $M$ and  $W_1$, $W_2$ two smooth subvarieties of $M$ contained in $H$ and transverse inside $H$. Consider now the composition $f$ of 
blow-ups 
$$ f: B:=\mathrm{Bl}_{\wh{W_2}} \mathrm{Bl}_{W_1} M \rTo^{f_2}  \mathrm{Bl}_{W_1} M \rTo^{f_1} M \;,$$where $\wh{W_2}$ is the strict transform of $W_2$ inside $\mathrm{Bl}_{W_1} M$. Denote with  $E_{W_1}$ the exceptional divisor  of $\mathrm{Bl}_{W_1} M$ 
and with $E_{\widehat{W_2}}$ that of $\mathrm{Bl}_{\wh{W_2}} \mathrm{Bl}_{W_1} M$. 
Then  $f$ is an isomorphism outside $f^{-1}(W_1 \cup W_2)$; moreover
$$ f^{-1}(\mc{I}_{W_1 \cup W_2}) = \mc{I}_{\widehat{E_{W_1}}} \cdot \mc{I}_{E_{\widehat{W_2}}} = \FS_B( -\widehat{E_{W_1}} - E_{\widehat{W_2}}) \;. 
$$Finally the relative canonical bundle $K_{B/M}$ is isomorphic to $\FS_B(\widehat{E_{W_1}} + E_{\widehat{W_2}})$. 
\end{lemma}
\begin{proof}In the particular case in which $M = \mbb{C}^3$; $\mc{I}_H = (x)$; $\mc{I}_{W_1} = (x, y)$; $\mc{I}_{W_2} = (x, z)$ and hence $\mc{I}_{W_1 \cup W_2} = (x, yz)$,  the statement can be proved by an explicit computation in coordinates, which we leave to the reader. 

Le't now pass to the general case. Consider a point $p$ in the intersection $W_1 \cap W_2$. Over an adequate open neighbourhood $U$ of $p$ in the standard complex topology, we can find local holomorphic coordinates $x, y, z$ such that $H$ is defined (over $U$) by the zeros of $x$, and 
$W_1$ and $W_2$ by the ideals $(x,y)$ and $(x,z)$, respectively.
Alternatively, one can find an adequate affine neighbourhood $U$ of $p$ and regular function $x, y, z$ over $U$ such that the differentials $dx, dy, dz$ are independent in $\mf{m}_q / \mf{m}_q^2$ for all $q \in U$ and such that $H$, $W_1$, $W_2$ are defined by ideals
of the regular functions $(x)$, $(x,y)$ and $(x, z)$ as in the holomorphic case. Hence the general situation can be  obtained locally from the particular one above by a smooth base change: the statement follows. 
\end{proof}

\begin{lemma}\label{lmm: locallemma}Let $M$ be a smooth algebraic variety, $H$ a smooth hypersurface of $M$, and $W$ and $Q$ two codimension 2 smooth subvarieties of $M$ such that $Q \subseteq H$, $W \cap H \subseteq Q$ and $W \cap H$ is a smooth 
codimension 3 subvariety of $M$.  Consider the blow-up $f: \mathrm{Bl}_W M \rTo M$ of 
$W$ in $M$, with exceptional divisor $E_W$. Then 
$$ f^{-1} (\mc{I}_W \cap \mc{I}_Q) = \mc{I}_{E_W} \cdot \mc{I}_{\widehat{Q}} = \mc{I}_{E_W} \cap \mc{I}_{\widehat{Q}}$$
where  $\widehat{Q}$
denote the strict transform of $Q$ 
in $\mathrm{Bl}_W M$. 
\end{lemma}
\begin{proof}The statement is local in nature, over the base $M$: hence, by placing ourserlves on a small open neighbourhood of a point 
$p \in W \cap H$ in the complex topology, equipped with some holomorphic coordinates $(x, y, z, w_1, \dots, w_r)$, we can suppose 
that the ideals of $H$, $W$ and $Q$ are given locally by $\mc{I}_H = (z)$, $\mc{I}_W = (x,y)$, $\mc{I}_Q = (x, z)$. Then $
\mc{I}_W \cap \mc{I}_Q = (x, yz)$;  
the proof of the statement 
is now achieved through an easy computation  in coordinates. 
\end{proof}

\subsection{A crepant resolution of $B^3$.}\label{subsection: crepant}Conjecture \ref{conj1} states that the log-canonical threshold of the pair 
$(X^3, \mc{I}_{\Delta_3})$ is $2$. This fact suggests that $B^3$ might admit  a crepant resolution. This is indeed the case, as we will prove in this subsection. 
\begin{remark}\label{rmk: universal}Let $X$ be a smooth algebraic surface. If $Y$ is any smooth variety admitting a  projective birational morphism $f: Y \rTo X^n$ over $X^n$ such that  such that 
$f^{-1}(\mc{I}_{\Delta_n})$ is an invertible ideal sheaf of $\FS_Y$, then, by the universal property of the blow-up, the map $f$ factors via 
the isospectral Hilbert scheme $B^n$ as 
\begin{diagram} Y & & \\
\dTo^h & \rdTo^{f} & \\ 
B^n & \rTo^p &  X^n 
\end{diagram}providing a resolution $h$ of $B^n$ such that 
$$ K_Y - h^* K_{B^n} = K_Y - h^* (p^* K_{X^n} + \FS_{B^n}(E)) = K_Y - f^* K_{X^n} + h^{-1} \mc{I}_E = K_Y - f^*K_{X^n} + f^{-1}(\mc{I}_{\Delta_n}) \;.$$ 
\end{remark}
\begin{remark}\label{rmk: universalcrepant}
By the previous remark, in order to find a crepant resolution of $B^n$, it is sufficient to build a smooth variety $Y$ and a  projective  birational map $f : Y \rTo X^n$ such that $f^{-1}(\mc{I}_{\Delta_n})$ is an invertible ideal isomorphic to the relative anticanonical 
$-K_{Y/X^n} =  f^*K_{X^n} -K_Y $. 
\end{remark}
\begin{remark}\label{rmk: C2}The questions posed in the previous two remarks are local over the base and analytical in nature. Hence, to find a resolution of $B^n$ in general, it is sufficient to find a smooth variety $Y$ and a birational map as in the remark \ref{rmk: universal} for $X = \mbb{C}^2$. Moreover, since 
in the identification (\ref{eq: identification}), the ideal sheaf $\mc{I}_{\Delta_n}$ corresponds to $\mc{I}_{\widetilde{D}_{n-1}} \boxtimes \FS_{\mbb{C}^2} $, by flat base change 
it is sufficient to find a smooth variety $Y$ and a projective birational morphism $f: Y \rTo (\mbb{C}^2)^{n-1}$ such that 
$f^{-1}(\mc{I}_{\widetilde{D}_{n-1}})$ is an invertible ideal. The resolution thus built will be crepant if and only if $f^{-1}(\mc{I}_{\widetilde{D}_{n-1}})$ is isomorphic to the anticanonical $-K_Y$. 
\end{remark}

For brevity's sake, in what follows, we will indicate the affine space $(\mbb{C}^2)^2$ with $V$, the subscheme $\widetilde{D}_2$ with 
$W$. Fix coordinates $(x, y, z, w)$ over $V$. The irreducible components of the subscheme $W$ are linear subspaces $W_1, W_2, W_3$, defined by the ideals $I_1 = (x,y)$, $I_2 = (z,w)$, $I_3 = (x-z, y-w)$. The ideal $\mc{I}_W$ is then given by $\langle q, I_1 I_2 I_3 \rangle$, where $q$ is the quadric $q = xw-yz$.  
\begin{pps}\label{pps: crepant}
The projective birational morphism $f: Y \rTo V$,  defined as the  composition of smooth blow-ups 
$$ Y = Y_3 \rTo^{f_3} Y_2 \rTo^{f_2} Y_1 \rTo^{f_1} V$$where $Y_1 = \mathrm{Bl}_{W_1} V$, $Y_2 = \mathrm{Bl}_{\wh{W}_2} Y_1$, 
$Y_3 = \mathrm{Bl}_{\wh{\wh{W}}_3} Y_2$, where $\wh{W}_2, \wh{\wh{W}}_3$ are the strict transforms of $W_2$, $W_3$ in $Y_1$, $Y_2$, respectively, is an isomorphism outside the locus $f^{-1}(W)$. Moreover, the ideal sheaf $f^{-1}(\mc{I}_W)$ is invertible and 
isomorphic to $-K_Y$. \end{pps}
\begin{proof}As generators of the ideal $\mc{I}_W$ we can choose the polynomials
$q, xz(x-z), xw(y-w), yw(x-z), yw(y-w)$.  
Consider the first blow-up 
$Y_1 = \mathrm{Bl_{W_1} V} \simeq \mathrm{Bl}_0(\mbb{C}^2) \times \mbb{C}^2$ and denote with $E_1$ the exceptional divisor. 
We can write globally $$ x = \lambda u \;,\qquad y = \lambda v $$where $\lambda$ is the canonical section of $\FS_{Y_1}(E_1)$ 
and $u,v$ are homogeneous coordinates, thought as a basis in $H^0(\FS_{Y_1}(-E_1))$. By definition of weak transform we have 
$ \mc{I}_{f_1^{-1}(W)} = \mc{I}_{E_1} \cdot \mc{I}_{\widetilde{W}} $. The weak transform $\widetilde{W}$ is given by
the equations \begin{gather*} uw - vz = 0 \\ 
uz (\lambda u - z) = 0 \\ 
uw (\lambda v-w) = 0 \\ 
vw (\lambda u - z) = 0 \\ 
vw (\lambda v-w) = 0 
\end{gather*}We prove now that the weak transform $\widetilde{W}$ concides with the strict transform $\wh{W}$. 
By proposition \ref{pps: lambda} and its proof we just have to show that the morphism $\lambda: \FS_{\widetilde{W}}(-E_1) \rTo  \FS_{\widetilde{W}}$ is injective. Now, $\widetilde{W}$ is contained in the hypersurface $H$ of $Y_1$ defined by the equation $uw-vz=0$. Over $H$ 
we can globally write $z = \mu u$, $w = \mu v$, where $\mu$ can be seen  as a section in $H^0(\FS_H(E_1))$. Then $\widetilde{W}$ is given, inside $H$, by the equations \begin{gather*}
u^3 \mu(\lambda - \mu) = 0 \\ 
u v^2 \mu( \lambda-\mu) = 0 \\ 
uv^2 \mu( \lambda - \mu) = 0 \\ 
v^3 \mu (  \lambda-\mu) = 0
\end{gather*}Since $u$ and $v$ do not vanish at the same time, the weak transform is given by the equation 
$\mu(\lambda-\mu) =0$ inside the hypersurface $H$, with respect to the coordinates $([u, v], \lambda, \mu)$. Hence $\lambda$ is not 
zero divisor in $\widetilde{W}$ and $\widetilde{W} = \widehat{W}$. Hence $$\mc{I}_{f_1^{-1}(W)} =  \mc{I}_{E_1} \cdot \mc{I}_{\wh{W}}  
\;.
$$Now $\widehat{W}$ is clearly the union, inside $H$, of the two smooth surfaces $\widehat{W_2}$ and $\widehat{W_3}$ intersecting transversally along a smooth curve inside the exceptional divisor $E_1$. 
Consider now the blow-ups $f_2: \mathrm{Bl}_{\wh{W}_2} Y_1 \rTo Y_1$, with exceptional divisor $E_2$, and $f_3: \mathrm{Bl}_{\wh{\wh{W}}_3 } Y_2 \rTo Y_2$, with exceptional divisor $E_3$; denote with $\wh{\wh{E_1}}$ and $\wh{E_2}$ the strict transforms of $E_1$ and $E_2$ in $Y_3$, respectively. Let now $g :=f_2 \circ f_3$ and let $f := f_1 \circ g$. Then by lemma \ref{lmm: w1w2} we have
\begin{align*} f^{-1}(\mc{I}_W) = & g^{-1}(\mc{I}_{f_1^{-1}(W)}) 
= g^{-1}(\mc{I}_{E_1}) \cdot g^{-1}(\mc{I}_{\wh{W}})  =  \mc{I}_{\wh{\wh{E_1}}} \cdot \mc{I}_{\wh{E}_2} \cdot \mc{I}_{E_3}    \;,
\end{align*}where we used that $\widetilde{E_1} = \widehat{E_1}$ and $\widetilde{\widehat{E_1}} = \widehat{\widehat{E_1}}$ by remark \ref{rmk: H}. Hence $f^{-1}(\mc{I}_W)$ is invertible and isomorphic to $\FS_{Y}(-\wh{\wh{E_1}} - \wh{E}_2 - E_3)$; it is now easy to show that the latter  coincides with the anticanonical 
divisor $-K_Y$. \end{proof}
As an immediate consequence of remarks \ref{rmk: universal}, \ref{rmk: universalcrepant} and \ref{rmk: C2} we deduce the
\begin{crl}\label{crl: crepant}The map $f: Y \rTo V$ factors through a crepant resolution $h: Y \rTo \mathrm{Bl}_W V$. Consequenty 
the map $h \times \id : Y \times \mbb{C}^2 \rTo \mathrm{Bl}_W V \times \mbb{C}^2 \simeq B^3$ identifies to a crepant resolution of $B^3$. 
\end{crl}
Let now $X$ be an arbitrary smooth algebraic surface and let $\Delta_{I_1}$, $\Delta_{I_2}$, $\Delta_{I_3}$ be the pairwise diagonals $\Delta_{I}$, $|I| =2$, 
taken in whatever order. We have the following 
\begin{theorem}The composition of blow-ups $s := s_1 \circ s_2 \circ s_3$
$$ Y:= \mathrm{Bl}_{\widehat{\widehat{\Delta}}_{I_3}} \!\! Y_2 \rTo^{s_3} Y_2:=\mathrm{Bl}_{\widehat{\Delta}_{I_2}} Y_1 \rTo^{s_2} Y_1 := \mathrm{Bl}_{\Delta_{I_1}} X^3 \rTo^{s_1} X^3 $$where $\widehat{\Delta}_{I_2}$ and $\widehat{\widehat{\Delta}}_{I_3}$ are the strict transforms of $\Delta_{I_2}$ and $\Delta_{I_3}$ in $Y_1$ and $Y_2$, respectively, is a log-resolution of the pair $(X^3, \mc{I}_{\Delta_3})$ such that $s^{-1}(\mc{I}_{\Delta_3})$ is an invertible ideal isomorphic to the relative anticanonical $-K_{Y/X^3}$. Hence $s$ factors through a crepant resolution $g: Y \rTo B^3$ of the isospectral Hilbert scheme~$B^3$. 
\end{theorem}\begin{proof}Locally over $X^3$, the map $s$ coincides precisely with $\phi^{-1} \circ (f \times \id_{\mbb{C}^2}) $, where $f$ is the birational map built in theorem
\ref{pps: crepant} and $\phi$ is the map (\ref{eq: change}). The theorem is then an immediate consequence of proposition \ref{pps: crepant} and remarks \ref{rmk: universal} and \ref{rmk: universalcrepant}. 
\end{proof}

\subsection{An $\perm_3$-equivariant resolution of $B^3$}
Consider the $4$-dimensional vector space $V = (\mbb{C}^2)^2$ with coordinates $(x,y,z,w)$ 
and the subscheme $W = W_1 \cup W_2 \cup W_3$ introduced in  subsection \ref{subsection: crepant}.  Consider the blow-up $f_1: Y_1 := \mathrm{Bl}_0 (V) \rTo V$ of $V$ at the origin and let $E_0$ be its 
 exceptional divisor; since it can be identified with the total space of the Hopf line bundle over the projective space $\mbb{P}(V)$, 
the variety $Y_1$ is equipped with a fibration $Y_1 \rTo \mbb{P}(V)$. 
Now, the polinomial $q = xw-yz$ defines a smooth quadric $Q$ in $\mbb{P}(V)$, which can be seen as a smooth subvariety of $Y_1$ inside $E_0$, thanks to the embedding of $\mbb{P}(V)$ into $Y_1$
given by 
 the zero section of the Hopf bundle. 
 \begin{pps}\label{pps: permlocalres} The birational morphism $f: Y \rTo V$ defined as the composition of smooth blow-ups 
$$ Y =Y_3 \rTo^{f_3} Y_2 \rTo^{f_2} Y_1 \rTo^{f_1} V $$where  
$Y_2 = \mathrm{Bl}_{\widehat{W}}(Y_1)$, 
$Y_3 = \mathrm{Bl}_{\widehat{Q}}(Y_2)$, 
where $\widehat{W}$ and $\widehat{Q}$ are the strict transforms of $W$ and $Q$ in $Y_1$ and $Y_2$, respectively, is an isomorphism outside $f^{-1}(W)$. Moreover the ideal sheaf $f^{-1}(\mc{I}_W)$ is given by 
$$ f^{-1}(\mc{I}_W) = \FS_Y(- 2 \widehat{\widehat{E_0}} -  \widehat{E_{\widehat{W}}} - 3 E_{\widehat{Q}} ) $$where $E_{\widehat{W}}$ and $E_{\widehat{Q}}$  are the exceptional divisors 
in $Y_2$ and $Y_3$, respectively, and where $\widehat{E_{\widehat{W}}}$ and $\widehat{\widehat{E_0}}$ are the strict transforms of $E_{\widehat{W}}$ and $E_0$ in $Y$. 
 \end{pps}
 
\begin{proof}Since $\ord_0 \mc{I}_W = 2$, we have 
$$ \mc{I}_{f_1^{-1}(W)} = \mc{I}_{E_0}^2 \cdot \mc{I}_{\widetilde{W}}$$where $\widetilde{W}$ is the weak transform of $W$ in $Y_1$. By a computation in coordinates, using the same generators for $\mc{I}_W$ we used in the proof of theorem \ref{pps: crepant}, 
one gets 
$$ \mc{I}_{\widetilde{W}} = \mc{I}_Q \cap \mc{I}_{\widehat{W}} \;,$$that is, the weak transform $\widetilde{W}$ is the scheme-theoretic union of the quadric $Q$ and the strict transform $\widehat{W}$ of $W$ in $Y_1$, which is a smooth codimension $2$ subvariety with three irreducible components 
$\widehat{W}_i$, $i = 1, \dots, 3$. Moreover $\widehat{W} \cap E_0$ is contained in $Q$ and is precisely the union of three skew lines in $ E_0 \simeq \mbb{P}(V)$; hence $\widehat{W} \cap E_0$ is a smooth codimension $3$ subvariety of $Y_1$. Therefore the hypothesis of lemma  \ref{lmm: locallemma} are satisfied; this means that, when blowing up the strict transform $\widehat{W}$ in $Y_1$ one gets 
$$ f_2^{-1}(\mc{I}_Q \cap  \mc{I}_{\widehat{W}} ) = \mc{I}_{E_{\widehat{W}}} \cdot \mc{I}_{\widehat{Q}} \;.$$Since $\ord_{\widehat{W}} \widehat{E_0} = 0$, we get
$$ (f_1 \circ f_2)^{-1} ( \mc{I}_{W} )= \mc{I}^2_{\widehat{E_0}} \cdot \mc{I}_{E_{\widehat{W}}} \cdot \mc{I}_{\widehat{Q}} \;.$$Remembering that 
$\ord_{\widehat{Q}} \widehat{E_0} =1$, the last blow-up now yields 
the formula in the statement. 
\end{proof}
\begin{crl}\label{crl: permlocalres}The map $f: Y \rTo V$ factors through a resolution $h: Y \rTo \mathrm{Bl}_W V$. Consequenty 
the map $h \times \id : Y \times \mbb{C}^2 \rTo \mathrm{Bl}_W V \times \mbb{C}^2 \simeq B^3$ identifies to an $\perm_3$-equivariant resolution of $B^3$. 
\end{crl}
Consider now the case of an arbitrary smooth algebraic surface $X$. Consider the blow-up $s_1: Y_1:=\mathrm{Bl}_{\Delta_{123}} X^3 \rTo X^3$ of the small diagonal $\Delta_{123}$ 
in $X^3$ and let $E_0$ be its exceptional divisor. The situation is locally, over $X^3$, analogous to the one just studied. Hence it is now clear that $s_1^{-1}(\mc{I}_{\Delta_3}) = \mc{I}_{E_0}^2 \cdot (\mc{I}_Q \cap \mc{I}_{\widehat{\Delta_3}})$, where $\widehat{\Delta_3}$ is the strict transform of $\Delta_3$ in $Y_1$ and where $Q$ is a quadric subbundle of $\mbb{P}(N_{\Delta_{123}/ X^3})$ over $\Delta_{123}$ and hence a smooth subvariety of $Y_1$ inside $E_0$. 
We have the following theorem
\begin{theorem}The composition of smooth blow-ups $s: =s_1 \circ s_2 \circ s_3$: 
$$ Y := \mathrm{Bl}_{\widehat{Q}} Y_2 \rTo^{s_3} Y_2 := \mathrm{Bl}_{\widehat{\Delta_3}} Y_1 \rTo^{s_2} Y_1 \rTo^{s_1} X^3 $$where $\widehat{\Delta}_3$ and $\widehat{Q}$ are the strict transforms of $\Delta_3$ and $Q$ in $Y_1$ and $Y_2$, respectively,  defines a $\perm_3$-equivariant log-resolution of the pair $(X^3, \mc{I}_{\Delta_3})$ and hence factors through a $\perm_3$-equivariant log-resolution $g: Y \rTo B^3$ of the isospectral Hilbert scheme $B^3$. 
\end{theorem}

\begin{proof}The map $s$ is clearly $\perm_3$-equivariant and, locally over $X^3$, coincides with the map $\phi^{-1} \circ (f \times \id_{\mbb{C}^2}) $, where $f$ is the map introduced in proposition \ref{pps: permlocalres} and where $\phi$ is the map (\ref{eq: change}). The content of the theorem is then 
a consequence of proposition \ref{pps: permlocalres}, corollary \ref{crl: permlocalres} and remarks \ref{rmk: universal} and \ref{rmk: universalcrepant}. 
\end{proof}

\begin{remark}This resolution is not crepant, as one gets easily $K_{Y/X^3} 
+ s^{-1}(\mc{I}_{\Delta_3}) = \FS(\widehat{\widehat{E_0}} + E_{\widehat{Q}})$, where 
$E_{\widehat{Q}}$ is the exceptional divisor in $Y_3$ and where
$\widehat{\widehat{E_0}}$ is the strict transform of $E_0$ in $Y$. 
\end{remark}
\begin{remark}The step $Y_2$ coincides with the Fulton-MacPherson compactification $X[3]$ of 
$X^3~\setminus~\Delta_3$ (see \cite{FultonMacPherson1994}). 
\end{remark}
\begin{remark}By construction, the resolution $Y$ is equipped with a $\perm_3$-action. The stabilizer of any point for this action is trivial. 
Hence, passing to the quotient modulo $\perm_3$, the induced map $\hat{f}: Y/\perm_3 \rTo S^3 X$ provides an explicit resolution of $S^3X$ which factors through 
the Hilbert scheme of points $X^{[3]} = B^3/ \perm_3 $. 
\end{remark}


\vspace{1cm}
\noindent
{Departamento de  Matem\'atica, Puc-Rio, Rua Marqu\^es S\~ao Vicente 225, 22451-900 G\'avea, Rio de Janeiro, RJ, Brazil}

\noindent 
{\it Email address:} {\tt lucascala@mat.puc-rio.br}

\end{document}